\documentclass[12pt]{amsart}
\usepackage{amsfonts, amsmath, amsthm, amssymb, amscd, amsrefs}
\usepackage{hyperref}
\usepackage{graphicx}
\usepackage[all]{xy}
\usepackage{tikz}
\usepackage{color,xcolor}

\usepackage{mathrsfs}
\theoremstyle{definition}

\theoremstyle{definition}    \newtheorem{dfn}{Definition}[section]
\theoremstyle{definition}    \newtheorem{rmk}[dfn]{Remark}
\theoremstyle{definition}    
\theoremstyle{definition}    \newtheorem{prp}[dfn]{Proposition}
\theoremstyle{definition}    
\theoremstyle{definition}    \newtheorem{lem}[dfn]{Lemma}
\theoremstyle{definition}    \newtheorem{cor}[dfn]{Corollary}
\theoremstyle{definition}    
\theoremstyle{definition}    \newtheorem{nta}[dfn]{Notation}
\theoremstyle{definition}    
\theoremstyle{definition}    
\theoremstyle{definition}    
\newcommand{\trm}{\textrm}
\newcommand{\tbf}{\textbf}

\newcommand{\C}{{\mathbb C}}
\newcommand{\R}{{\mathbb R}}

\newcommand{\Z}{{\mathbb Z}}

  \newcommand{\cC}{\mathcal C}
\newcommand{\cD}{\mathcal D} 
\newcommand{\cF}{\mathcal F}  

\newcommand{\cL}{\mathcal L}

\newcommand{\cU}{\mathcal U} \newcommand{\cV}{\mathcal V} 
\newcommand{\f}{\mathfrak}

\newcommand{\al}{\alpha}
\newcommand{\be}{\beta}
\newcommand{\ga}{\gamma}     \newcommand{\Ga}{{\Gamma}}
\newcommand{\de}{\delta}

\newcommand{\la}{\lambda}    \newcommand{\La}{{\Lambda}}
\newcommand{\te}{\theta}    
\newcommand{\si}{\sigma}     
\newcommand{\om}{\omega}     \newcommand{\Om}{{\Omega}}
\newcommand{\vph}{\varphi}
\newcommand{\ptm}{\phantom{1}}

\newcommand{\vect}{\textrm{Vect}}

\newcommand{\tr}{\textrm{tr}}
\newcommand{\ch}{\textrm{ch}}

\newcommand{\Ch}{\textrm{Ch}}
\newcommand{\cs}{\textrm{cs}}
\newcommand{\CS}{\textrm{CS}}

\newcommand{\odd}{\textrm{odd}}
\newcommand{\even}{\textrm{even}}

\newcommand{\we}{\wedge}      

\newcommand{\dsum}{\oplus}    
      
\newcommand{\srl}{\stackrel}
 \newcommand{\ra}{\rightarrow} \newcommand{\lra}{\longrightarrow} 

  \newcommand{\wbar}{\overline}

\newcommand{\wtl}{\widetilde}

\newcommand{\ptl}{\partial}
\newcommand{\na}{\nabla}

\newcommand{\bl}{\bullet}
\newcommand{\isom}{\cong}     
\newcommand{\bmat}{\left(\begin{array}}  \newcommand{\emat}{\end{array}\right)}
\newcommand{\barr}{\begin{array}}  \newcommand{\earr}{\end{array}}
\newcommand{\bcd}{\begin{CD}}  \newcommand{\ecd}{\end{CD}}
\newcommand{\beq}{\begin{equation}\begin{aligned}}  \newcommand{\eeq}{\end{aligned}\end{equation}}
\newcommand{\beqs}{\begin{equation*}\begin{aligned}}  \newcommand{\eeqs}{\end{aligned}\end{equation*}}
\title{Geometric models of twisted differential $K$-theory I}
\author{Byungdo Park}
\address{Hausdorff Research Institute for Mathematics (HIM), Poppelsdorfer Allee 45, 53115 Bonn, Germany}

\email{byungdpark@gmail.com}
\date{2 August 2017}



\begin{document}
\sloppy
\maketitle
\begin{abstract}
This is the first in a series of papers constructing geometric models of twisted differential $K$-theory. In this paper we construct a model of even twisted differential $K$-theory when the underlying topological twist represents a torsion class. By differential twists we will mean smooth $U(1)$-gerbes with connection, and we use twisted vector bundles with connection as cocycles. The model we construct satisfies the axioms of Kahle and Valentino, including functoriality, naturality of twists, and the hexagon diagram.  This paper confirms a long-standing hypothetical idea that twisted vector bundles with connection define twisted differential $K$-theory.

\end{abstract}
\maketitle
\setcounter{tocdepth}{1}
\tableofcontents
\section{Introduction}

There has been a considerable interest in twisted and differential refinements of generalized cohomology theories. Several aspects of topology, geometry, analysis, and physics coalesce in this area, and it has provided interesting applications of $\infty$-categorical machinery nicely packaged by $\infty$-sheaves of spectra on the site of smooth manifolds. (See \cites{BN,BNV} for example.)

One important question in generalized cohomology theories is whether one can represent an element of a given generalized cohomology theory of a space using geometric cocycles. For instance, an element of the singular cohomology group of a space can be represented by a singular cocycle in the space, and elements in the complex $K$-theory of a space can be represented by complex vector bundles over that space. However, geometric models are still unknown for most other cohomology theories such as topological modular forms, and even less is known for their twisted and differential refinements. This paper is the first in a series of papers to answer this question for the case of differential refinements of twisted complex $K$-theory.

Twisted $K$-theory was first introduced by Donovan and Karoubi in \cite{DK}, where twists represent torsion classes in degree $3$ integral cohomology, and Rosenberg \cite{Ro1} for all classes. More recently twisted $K$-theory has received much attention because of its applications in classifying D-brane charges in string theory \cite{Wi}, Verlinde algebras \cite{FHT}, and topological insulators \cite{FM}.

An archetype of differential $K$-theory first appeared in Karoubi \cite{Ka1} as the \emph{multiplicative $K$-theory}. This is nowadays known as the flat subgroup of the differential $K$-theory. Lott \cite{Lo} used Karoubi's construction to develop an index theorem, but it was mostly applications in string theory that have rekindled a considerable interest in differential $K$-theory (see Freed \cite{F} for example). Perhaps one of the most remarkable steps forward in differential cohomologies is due to Hopkins and Singer \cite{HS} wherein they construct a differential extension of any exotic cohomology theory in a homotopy-theoretic way. Following this work, Bunke and Schick \cite{BS}, Klonoff, Freed, and Lott \cites{Kl, FL}, Simons and Sullivan \cite{SS}, and Tradler, Wilson, and Zeinalian \cites{TWZ, TWZ2} all came up with more concrete and geometric models of differential $K$-theory. For a more detailed survey on recent developments of differential $K$-theory, we refer the reader to Bunke and Schick \cite{BS2}.

There have been some attempts to twist differential $K$-theory. Carey, Mickelsson, and Wang \cite{CMW} gave a construction that satisfies the square diagram and short exact sequences. Kahle and Valentino \cite{KV}, in an attempt to precisely formulate the $T$-duality for Ramond-Ramond fields in the presence of a $B$-field, gave a list of axioms for twisted differential $K$-theory, which can be generalized as axioms for any twisted differential cohomology theory. They construct a canonical differential twist for the differential $K$-theory of the total space of any torus bundle (\cite{KV} Section 2.2).  However, a construction of twisted differential $K$-theory that satisfies Kahle-Valentino axioms had not been found until very recently: In 2014, Bunke and Nikolaus \cite{BN} constructed a differential refinement of any twisted cohomology theory. Their construction of twisted differential $K$-theory satisfies several properties we would expect, including all Kahle-Valentino axioms, except the push-forward axiom which is not addressed in \cite{BN}. The construction of Bunke and Nikolaus provides a correct model for twisted differential cohomology theory in that their model combines twisted cohomology groups and twisted differential forms in a homotopy theoretic way, analogous to what Hopkins and Singer did in the untwisted case. However, the Bunke-Nikolaus model is not very geometric just as the Hopkins-Singer model is not. We might hope that there exists a more geometric model for a twisted differential cohomology theory, at least in the case of $K$-theory.

The goal of this paper is to construct such a geometric model of twisted differential even $K$-theory in the case that the underlying topological twists represent torsion classes in degree $3$ integral cohomology. We use $U(1)$-gerbes with connection as differential twists and twisted vector bundles with connection as cycles.  We also have constructed a twisted differential $K$-theory for both torsion and non-torsion twistings using lifting bundle gerbes with connection and curving as differential twists and $U_{\tr}$-bundle gerbe modules with connection (due to Bouwknegt, Carey, Mathai, Murray, and Stevenson \cite{BCMMS}) as cycles, which will be discussed in the second paper of this series. Both of our models satisfy all of the Kahle-Valentino axioms except the push-forward axiom which, together with a model of odd twisted differential $K$-theory, will be discussed in subsequent papers.

This paper is organized as follows. In Section \ref{SEC.2}, we review twisted vector bundles and set up some notation. Section \ref{SEC.3} constructs the twisted Chern character form and the twisted Chern-Simons form. We also verify several properties which will be needed in later sections. Section \ref{SEC.4} defines differential twists and constructs an even twisted differential $K$-group. We then show that our construction is functorial, natural with respect to change of differential twist, define maps into and out of the twisted differential $K$-groups, and verify that our model fits into a twisted analogue of the differential $K$-theory hexagon diagram \`a la Simons and Sullivan \cite{SS}.

Having constructed geometric models of the even twisted differential $K$-theory, a natural question arises: ``Is there a map between our geometric model and the Bunke-Nikolaus model?'' Bunke, Nikolaus and V\"olkl \cite{BNV}  answered this question for the case of untwisted differential $K$-theory. In this case, there is a way to obtain a sheaf of spectra on the site of smooth manifolds using the symmetric monoidal category of vector bundles with connection. In \cite{BNV}, they obtain a map between this sheaf of spectra into a Hopkins-Singer sheaf of spectra by the universal property of the pullback. The induced map between abelian groups is called a \emph{cycle map}. Constructing a twisted analogue of the cycle map along this vein is work in progress, which we hope to complete in the near future.

\noindent\textbf{Acknowledgements.} We would like to thank Mahmoud Zeinalian for many helpful discussions and comments as well as constant support during the entire process of this work. We also would like to thank Scott Wilson, Arthur Parzygnat, and Corbett Redden for reading the preliminary version of this paper and providing helpful suggestions and comments as well as Thomas Tradler and Cheyne Miller for useful conversations. We thank Jim Stasheff, Hisham Sati, Ping Xu, and Mathieu Sti\'enon for their interest and detailed comments on this work. We gratefully acknowledge a partial support on this work from Mahmoud Zeinalian's NSF grant DMS-1309099 as well as support and hospitality from the Hausdorff Research Institute for Mathematics during our visit. Finally, we thank the anonymous referee for detailed comments.

\section{Review of twisted vector bundles and twisted $K$-theory}\label{SEC.2}
In this section, we set up notations and briefly review $\la$-twisted vector bundles. A good reference on twisted vector bundles is Karoubi \cite{Ka2}, which has a broader account.

\begin{nta}
Throughout this paper, all of our manifolds are connected compact smooth manifolds, and all our maps are smooth maps unless specified otherwise. In particular, $X$ and $Y$ always denote manifolds. We will use the notation  $U_{i_1\cdots i_n}$ to denote an $n$-fold intersection $U_{i_1}\cap\cdots\cap U_{i_n}$. If an open cover is locally finite and every $n$-fold intersection is contractible for all $n\in\Z^+$, we will call it a good cover.

\end{nta}

\begin{dfn}
Let $\cU=\{U_i\}_{i\in  I}$ be an open cover of $X$, and $\la$ be a $U(1)$-valued completely normalized \v{C}ech $2$-cocycle as recalled below. A \textbf{$\la$-twisted vector bundle} $E$ of rank $n$ over $X$  consists of a family of product bundles $\{U_i\times \C^n:U_i\in \cU\}_{i\in \La}$ together with transition maps $$g_{ji}:U_{ij}\ra U(n)$$ satisfying $$g_{ii}=\tbf{1},\quad g_{ji}=g_{ij}^{-1},\quad g_{kj}g_{ji}=g_{ki}\la_{kji}.$$
\end{dfn}

\begin{rmk} (1) Recall that a \v{C}ech cocycle $\zeta=(\zeta_{i_1\cdots i_n})$ is called \emph{completely normalized} if $\zeta_{i_1\cdots i_n}\equiv 1$ whenever there is a repeated index, and $\zeta_{\si(i_1)\cdots \si(i_n)}=(\zeta_{i_1\cdots i_n})^{\trm{sign}(\si)}$ for any $\si\in\mathfrak{S}_n$, where $\mathfrak{S}_n$ is the symmetric group on $n$ letters.

(2) We write a $\la$-twisted vector bundle $E$ of rank $n$ as a triple $(\cU,\{g_{ji}\},\{\la_{kji}\})$, or a pair $(\{g_{ji}\},\{\la_{kji}\})$ if the open cover $\cU$ is clear from the context. When the rank $n$ is zero, there exists a $\la$-twisted vector bundle $\mathscr{O}=(\{g_{ji}^{\mathscr{O}}\},\{\la_{kji}\})$ with $g_{ji}^{\mathscr{O}}=\tbf{1}$ for all $i,j\in \La$. We call it the \emph{zero $\la$-twisted vector bundle}.
\end{rmk}

\begin{dfn}
A \textbf{morphism} $f$ from a $\la$-twisted vector bundle $E=(\{g_{ji}\},\{\la_{kji}\})$ of rank $n$ to a $\la$-twisted vector bundle $F=(\{h_{ji}\},\{\la_{kji}\})$ of rank $n$, with respect to the same open cover $\{U_i\}_{i\in  I}$ of the base $X$, is a family of maps $\{f_i:U_i\ra U(n)\}_{i\in \La}$ such that $$f_j(x)g_{ji}(x)=h_{ji}(x)f_i(x)\qquad\trm{for all }x\in U_{ij}\trm{ and all } i,j\in\La.$$
\end{dfn}

\begin{dfn} Let $E=(\{{g}_{ji}\},\{{\la}_{kji}\})$ and $F=(\{{h}_{ji}\},\{{\la}_{kji}\})$ be $\la$-twisted vector bundles of rank $n$ and $m$ with respect to the same covering $\cU=\{U_i\}$ of $X$. The \textbf{direct sum} $E\dsum F$ is defined by $(\{{g}_{ji}\dsum {h}_{ji}\},\{{\la}_{kji}\})$, and is a $\la$-twisted vector bundle of rank $n+m$. The symbol $\dsum$ between two transition maps denotes the block sum of matrices.
\end{dfn}

We denote the category of $\la$-twisted vector bundles over $X$ defined on an open cover $\cU$ by $\tbf{Bun}(\cU,\la)$. The category $\tbf{Bun}(\cU,\la)$ is an additive category with respect to the direct sum $\dsum$.

\begin{dfn} The \textbf{twisted $K$-theory} of $X$ defined on an open cover $\cU$ with a $U(1)$-gerbe twisting $\la$, denoted by $K^0(\cU,\la)$, is the Grothendieck group of the commutative monoid $\vect(\cU,\la)$ of isomorphism classes of $\la$-twisted vector bundles on $\cU$.
\end{dfn}

\begin{rmk}\label{RMK.canon.isom.two.different.twisting.K.gp} The isomorphism class of the group $K^0(\cU,\la)$ for a fixed $\cU$ depends only on the \v{C}ech cohomology class of $\la$. To see this, let $\cC$ and $\cD$ be additive categories. Recall that if an additive functor $F:\cC\ra \cD$ is an equivalence of additive categories, then it induces an isomorphism of groups $F_*:K(\cC)\ra K(\cD)$, where $K$ denotes the $K$-theory functor from additive categories to abelian groups. Let $\si$ and $\la$ be cohomologous $U(1)$-valued \v{C}ech 2-cocycles defined on $\cU$, i.e., $\la_{kji}=\si_{kji}\chi_{ji}\chi_{ik}\chi_{kj}$ for some \v{C}ech 1-cochain $\chi$. There is an additive functor $\Phi:\tbf{Bun}(\cU,\si)\ra \tbf{Bun}(\cU,\la)$ that is an isomorphism of categories. The functor $\Phi$ takes a $\si$-twisted vector bundle $(\{g_{ji}\},\{\si_{kji}\})$ to a $\la$-twisted vector bundle $(\{g_{ji}\chi_{ji}\},\{\la_{kji}\})$ and takes a morphism between $\si$-twisted vector bundles to itself. The inverse of $\Phi$ is defined similarly by taking a $\la$-twisted vector bundle $(\{g_{ji}\},\{\la_{kji}\})$ to the $\si$-twisted vector bundle $(\{g_{ji}\chi_{ji}^{-1}\},\{\si_{kji}\})$. Therefore, the induced map of groups $\Phi_*:K^0(\cU,\si)\ra K^0(\cU,\la)$ is an isomorphism.
\end{rmk}

\begin{dfn} Let $f:Y\ra X$ be a map, and $E=(\{{g}_{ji}\},\{{\la}_{kji}\})$ be a $\la$-twisted vector bundle defined on a covering $\cU=\{U_i\}$ of $X$. Let $f^{-1}\cU$ denote the open cover on $Y$ consisting of open sets of the form $f^{-1}(U_i)$. The \textbf{pull-back} of the $\la$-twisted vector bundle $E$ is a $(\la\circ f)$-twisted vector bundle $(f^{-1}\cU,\{{g}_{ji}\circ f\},\{{\la}_{kji}\circ f\})$ on $Y$ denoted by $f^*(E)$.
\end{dfn}

\begin{prp} The map \beqs f^*:\vect(\cU,\la)&\ra \vect(f^{-1}\cU,\la\circ f)\\
[E] &\mapsto [f^*E]
\eeqs is a monoid homomorphism with respect to $\dsum$ and therefore induces a group homomorphism \beqs f^*:K^0(\cU,\la)&\ra K^0(f^{-1}\cU,\la\circ f)\\
[E]-[F] &\mapsto [f^*E]-[f^*F].
\eeqs
\end{prp}
\begin{proof} The map is well-defined on $\vect(\cU,\la)$. Given another $\la$-twisted vector bundle $F$, we have $f^*(E)\dsum f^*(F)= f^*(E\dsum F)$. Hence $f^*$ is a monoid homomorphism, which induces a group homomorphism $f^*$ between $K$-groups.
\end{proof}

\section{Chern-Weil theory of twisted vector bundles}\label{SEC.3}
In this section, we review Chern-Weil theory of twisted vector bundles and define twisted Chern-Simons forms. We will also prove several lemmata which will be needed in subsequent sections. We begin with a summary of the language of $U(1)$-gerbes with connection used in this paper. We refer the reader to Gaw{\polhk{e}}dzki and Reis \cite{GR} for more details.

\begin{dfn}\label{DFN.GerbeWithConnection} Let $X$ be a   manifold and  $\cU:=\{U_i\}_{i\in \La}$ an open cover of $X$. A \textbf{$U(1)$-gerbe} over $X$ subordinate to $\cU$ is a $U(1)$-valued completely normalized \v{C}ech $2$-cocycle $\{\la_{kji}\}\in \check Z^2(\cU,U(1))$. A \textbf{connection} on a $U(1)$-gerbe $\{\la_{kji}\}$ on $\cU$ is a pair $(\{A_{ji}\},\{B_i\})$ consisting of a family of differential $1$-forms $\{A_{ji}\in \Om^1(U_{ij};\sqrt{-1}\R)\}_{i,j\in \La}$ and a family of differential $2$-forms $\{B_i\in \Om^2(U_{i};\sqrt{-1}\R)\}_{i\in \La}$ satisfying the following relations: \begin{itemize}
\item[\textbf{C1.} ] $\la_{kji}\la_{lji}^{-1}\la_{lki}\la_{lkj}^{-1}=1$
\item[\textbf{C2.} ] $d\log\la_{kji}=A_{ji}+A_{ik}+A_{kj}$
\item[\textbf{C3.} ] $B_{j}-B_{i}=dA_{ji}$
\end{itemize}
\end{dfn}
\begin{rmk}\label{RMK.U1gerbe.and.curvatureform} (1) A $U(1)$-gerbe with connection on $\cU$ is therefore a Deligne cocycle of degree $2$. Notice that our total differential is $D=d+(-1)^{q}\de$ on $\check{C}^p(\cU,\Om^q)$.

(2) From $dB_i=dB_j$ for all $i,j\in \La$, the family of exact $3$-forms $\{dB_i\}_{i\in \La}$ defines a global closed differential $3$-form $H$. The differential form $H$ is called the \emph{curvature} of the $U(1)$-gerbe or the \emph{Neveu-Schwarz $3$-form}.

(3) Let $\{\la_{kji}\}\in \check Z^2(\cU,U(1))$ be a $U(1)$-gerbe, and $\de:\check H^2(\cU,U(1))\ra H^3(X;2\pi i\Z)$ be the connecting map. The image in $H_{\trm{dR}}^3(X;\sqrt{-1}\R)$ of the cohomology class $\de([\la])\in H^3(X;2\pi i\Z)$ coincides with the cohomology class of $H\in H_{\trm{dR}}^3(X;\sqrt{-1}\R)$. (See Brylinski \cite[p.175]{Br} Corollary 4.2.8.)
\end{rmk}

Throughout the rest of this paper $\widehat\la=(\{\la_{kji}\},\{A_{ji}\},\{B_i\})$ always denotes a $U(1)$-gerbe with connection defined on an open cover $\cU=\{U_i\}_{i\in \La}$ of $X$ and $H$ denotes the $3$-curvature form of $\widehat\la$. We assume that the Dixmier-Douady class of $\la$ is a torsion class in $H^3(X;\Z)$.

\begin{dfn}\label{DEF.twisted.connection} Let  $\widehat\la=(\{\la_{kji}\},\{A_{ji}\},\{B_i\})$ be as above, $E=(\cU,\{g_{ji}\},\{\la_{kji}\})$ a smooth $\la$-twisted vector bundle of rank $n$. A \textbf{connection} on $E$ compatible with $\widehat\la$ is a family $\Ga=\{\Ga_i\in \Om^1(U_i;\mathfrak{u}(n))\}_{i\in\La}$  satisfying  \beq\label{EQN.twisted.connection.local.rel} \Ga_i-g_{ji}^{-1}\Ga_j g_{ji}-g_{ji}^{-1}dg_{ji}=-A_{ji}\cdot \tbf{1},
\eeq  where $A_{ji}\in \Om^1(U_{ij};i\R)$. Here $\mathfrak{u}(n)$ denotes the Lie algebra of $U(n)$, and \tbf{1} the identity matrix.
\end{dfn}

\begin{lem}\label{LEM.conn.off.term.rel} In the notation of Definition \ref{DEF.twisted.connection}, $A_{ji}\cdot\tbf{1}-A_{ki}\cdot\tbf{1}+A_{kj}\cdot\tbf{1}=\la_{kji}^{-1}d\la_{kji}\cdot\tbf{1}$.
\end{lem}
\begin{proof}\beqs (A_{ji}-A_{ki}+A_{kj})\cdot\tbf{1}&=-(\Ga_{i}-g^{-1}_{ji}\Ga_{j}g_{ji}-g^{-1}_{ji}dg_{ji})+(\Ga_{i}-g^{-1}_{ki}\Ga_{k}g_{ki}-g^{-1}_{ki}dg_{ki})+A_{kj}\cdot\tbf{1}\\
&=g^{-1}_{ji}(g^{-1}_{kj}\Ga_{k}g_{kj}+g^{-1}_{kj}dg_{kj}-A_{kj}\cdot\tbf{1})g_{ji}+g^{-1}_{ji}dg_{ji}\\&\qquad-\la_{kji}g_{ij}g_{jk}\Ga_{k}g_{kj}g_{ji}\la_{kji}^{-1}-g^{-1}_{ki}dg_{ki}+A_{kj}\cdot\tbf{1}\\
&=g^{-1}_{ji}g^{-1}_{kj}\Ga_{k}g_{kj}g_{ji}+g^{-1}_{ji}g^{-1}_{kj}dg_{kj}g_{ji}-A_{kj}\cdot\tbf{1}+g^{-1}_{ji}dg_{ji}\\&\qquad-g_{ij}g_{jk}\Ga_{k}g_{kj}g_{ji}-g^{-1}_{ki}dg_{ki}+A_{kj}\cdot\tbf{1}\\
&=g^{-1}_{ji}g^{-1}_{kj}(-g_{kj}dg_{ji}+dg_{ki}\la_{kji}+g_{ki}d\la_{kji})+g^{-1}_{ji}dg_{ji}-g^{-1}_{ki}dg_{ki}\\
&=g_{ik}g_{kj}g_{ji}g^{-1}_{ji}g^{-1}_{kj}dg_{ki}+g^{-1}_{ji}g^{-1}_{kj}g_{ki}d\la_{kji}-g^{-1}_{ki}dg_{ki}=\la_{kji}^{-1}d\la_{kji}\cdot\tbf{1}.
\eeqs
\end{proof}

\begin{rmk} For any $\la$-twisted vector bundle $E$, there exists a connection on $E$ associated with $\widehat\la$. See \cite[p.244]{Ka2}.
\end{rmk}

\begin{dfn}\label{DEF.curvature.twisted.case} Let  $\widehat\la=(\{\la_{kji}\},\{A_{ji}\},\{B_i\})$ be as above, $(E,\Ga)$ a $\la$-twisted vector bundle $(\cU,\{g_{ji}\},\{\la_{kji}\})$ of rank $n$ with a connection $\Ga$ compatible with $\widehat\la$. The \textbf{curvature form} of $\Ga$ is the family $R=\{R_i\in \Om^2(U_i;\mathfrak{u}(n))\}_{i\in\La}$, where $R_i:=d\Ga_i+\Ga_i\we \Ga_i$.
\end{dfn}

\begin{lem} For each $m\in\Z^+$, the differential forms $\tr\left[(R_i-B_i\cdot\tbf{1})^m\right]$ over the open sets $U_i$ glue together to define a global differential form on $X$. Here $B_i\in \Om^2(U_i;i\R)$ is given by $\widehat{\la}$.
\end{lem}

\begin{proof} From \eqref{EQN.twisted.connection.local.rel}, it follows that $R_i=g^{-1}_{ji}R_jg_{ji}-dA_{ji}\cdot\tbf{1}$. Then
\beqs
 \tr\left[(R_i-B_i\cdot\tbf{1})^m\right]&=\tr\left[(g_{ji}^{-1}R_jg_{ji}-dA_{ji}\cdot\tbf{1}-B_i\cdot\tbf{1})^m\right]=\tr\left[(g_{ji}^{-1}R_jg_{ji}-B_j\cdot\tbf{1})^m\right]\\
&=\tr\left[{\sum_{r=0}^m} {{m} \choose {r}} g_{ji}^{-1}R_j^{m-r}g_{ji}(-1)^r(B_j\cdot\tbf{1})^r\right]\\
&\srl{*}{=}\tr\left[{\sum_{r=0}^m} {{m} \choose {r}} R_j^{m-r}(-1)^r(B_j\cdot\tbf{1})^r\right]=\tr\left[(R_j-B_j\cdot\tbf{1})^m\right],
\eeqs
where ${{m} \choose {r}}$ is the binomial coefficient $m$ choose $r$, and at $*$, we have used $\tr(AB)=\tr(BA)$ and the fact that $B_j\cdot\tbf{1}$ commutes with other matrices.
\end{proof}

\begin{dfn} Let  $\widehat\la=(\{\la_{kji}\},\{A_{ji}\},\{B_i\})$ be as above, $H$ the $3$-curvature of $\widehat\la$,
and $(E,\Ga)$ a $\la$-twisted vector bundle with connection compatible with $\widehat\la$. For $m\in\Z^+$, the $m^{\trm{th}}$ \textbf{twisted Chern character form} is defined by
$$\ch_{(m)}(\Ga)(x):= \tr(R_i(x)-B_i(x)\cdot\tbf{1})^m \qquad\textrm{ $x\in U_{i}$}.$$ When $m=0$, define $\ch_{(0)}(\Ga)$ to be the rank of $E$.
The \textbf{total twisted Chern character form} is defined by
$$\ch(\Ga):=\trm{rank}(E)+\sum_{m=1}^\infty \frac{1}{m!}\ch_{(m)}(\Ga),$$ which will be sometimes denoted by $\ch(E,\Ga)$.
\end{dfn}

\begin{rmk} Recall that the $\Z_2$-graded sequence of differential forms $\cdots \ra\Om^{\even}(X)\srl{\small{d+H}}\lra \Om^{\odd}(X)\srl{\small{d+H}}\lra\cdots$ is a complex if $H$ is a closed $3$-form on $X$.  Then the \emph{twisted de Rham cohomology} of $X$ is the cohomology of this complex and  denoted by $H_H^{\trm{even/odd}}(X)$. If closed $3$-forms $H$ and $H'$ are cohomologous, i.e. $H'=H+d\xi$, the multiplication by $\exp(\xi)$ induces an isomorphism $H_H^{\bl}(X)\ra H_{H'}^{\bl}(X)$. We refer the reader to Atiyah and Segal \cite{ASe} for more details on twisted cohomology.
\end{rmk}

The following fact is well-known. (See \cite[p.29]{BCMMS} for example.)

\begin{prp} For each $m\in\Z^+$, $$d\ch_{(m)}(\Ga)=m\ch_{(m-1)}(\Ga)H.$$ Hence the total twisted Chern character form $\ch(\Ga)$ is $(d+H)$-closed.
\end{prp}

\begin{prp}\label{PRP.ch.is.additive} The $m^{\trm{th}}$ twisted chern character form is additive for all $m\geq 0$, i.e., $$\ch_{(m)}(\Ga^E\dsum \Ga^F)=\ch_{(m)}(\Ga^E)+\ch_{(m)}(\Ga^F).$$
\end{prp}

\begin{dfn} Let  $\widehat\la=(\{\la_{kji}\},\{A_{ji}\},\{B_i\})$ be as above, and $(E,\Ga)$ a $\la$-twisted vector bundle $(\cU,\{g_{ji}\},\{\la_{kji}\})$ of rank $n$ with a connection $\Ga$ compatible with $\widehat\la$. Let $f:(Y,\cV) \ra (X,\cU)$ be any map provided that $\cV=f^{-1}\cU$.  The \textbf{pullback} of $(E,\Ga)$ along $f$ is $f^*(E)$ together with the family $$f^*\Ga:=\{f^*\Ga_i\}_{i\in\La},$$ where $f^*\Ga_i\in\Om^1(f^{-1}(U_i);\mathfrak{u}(n))$ is defined by entrywise pullback.
\end{dfn}

\begin{prp} $f^*\Ga$ is a connection on the $(\la\circ f)$-twisted vector bundle $(\cV,\{g_{ji}\circ f\},\{\la_{kji}\circ f\})$ of rank $n$ compatible with $f^*\widehat\la=(\{\la_{kji}\circ f\},\{f^*A_{ji}\},\{f^*B_i\})$ and $f^*\ch(\Ga)=\ch(f^*\Ga).$
\end{prp}

\begin{prp}\label{PRP.naturality.ch.bundleisom}
Let  $\widehat\la=(\{\la_{kji}\},\{A_{ji}\},\{B_i\})$ be as above and $\vph:E\ra F$ an isomorphism of  $\la$-twisted vector bundles over $X$ with respect to the same open cover $\cU$.  Let $\Ga^E$ be a connection on $E$ associated with $\widehat\la$ and $\Ga^F$ a connection on $F$ associated with $\widehat\la$. Then $\ch(\Ga^F)=\ch(\vph^*\Ga^F)$.
\end{prp}

\begin{rmk}
We shall prove in Proposition \ref{PRP.twisted.Chern-Weil} that the total \emph{twisted Chern character} in twisted de Rham cohomology group is independent of the choice of connection.
\end{rmk}

\begin{prp}\label{PRP.naturality.of.twist.chern.char.the.same}
Let $\widehat\la=(\{\la_{kji}\},\{A_{ji}\},\{B_i\})$ and $\widehat\la'=(\{\la'_{kji}\},\{A'_{ji}\},\{B'_i\})$ be two $U(1)$-gerbes with connection defined on an open cover $\cU=\{U_i\}_{i\in\La}$ over $X$. Suppose $\widehat\la$ and $\widehat\la'$ are cohomologous as Deligne $2$-cocycles such that $\widehat\la'=\widehat\la+D\widehat\al$, where $\widehat\al=(\{\chi_{ji}\},\{\Pi_i\})\in \check C^1(\cU,\Om^1)$. (See Remark \ref{RMK.U1gerbe.and.curvatureform} for the definition of $D$.) Let $E=(\cU,\{g_{ji}\},\{\la_{kji}\})$ be a $\la$-twisted vector bundle of rank $n$ and $\Ga=\{\Ga_i\}_{i\in\La}$ a connection on $E$ compatible with $\widehat\la$. Define a $\la'$-twisted vector bundle $E'$ with connection $\Ga'$ compatible with $\widehat\la'$ by \beqs
E' &:= (\cU,\chi_{ji}g_{ji},\la'_{kji})\\
\Ga' &:= \{\Ga'_i\}_{i\in\La},\trm{ where }\Ga'_i:=\Ga_i+\Pi_i\cdot\tbf{1}.\\
\eeqs Then $\ch(\Ga)=\ch(\Ga').$
\end{prp}

\begin{rmk} Since $\widehat\la$ and $\widehat\la'$ are cohomologous, their $3$-curvatures are the same.
\end{rmk}

\begin{proof}[Proof of Proposition \ref{PRP.naturality.of.twist.chern.char.the.same}] From \beqs
R_i'=d\Ga'_i+\Ga'_i\we\Ga'_i&=d(\Ga_i+\Pi_i\cdot\tbf{1})+(\Ga_i+\Pi_i\cdot\tbf{1})\we(\Ga_i+\Pi_i\cdot\tbf{1})\\
&=d\Ga_i+d\Pi_i\cdot\tbf{1}+\Ga_i\we\Ga_i+\Ga_i\we\Pi_i\cdot\tbf{1}+\Pi_i\cdot\tbf{1}\we\Ga_i+\Pi_i\cdot\tbf{1}\we\Pi_i\cdot\tbf{1}\\
&=R_i+d\Pi_i\cdot\tbf{1},
\eeqs it follows that \beqs
\ch_{(m)}(E,\Ga)-\ch_{(m)}(E',\Ga')&=\tr(R_i-B_i\cdot\tbf{1})^m-\tr(R'_i-B'_i\cdot\tbf{1})^m=0
\eeqs since $B_i'-B_i=d\Pi_i$.
\end{proof}

\begin{nta}\label{NTA.change.of.curving.by.a.global.2.form} Given $\widehat\la$ and $\xi\in\Om^2(X;i\R)$, we denote by $\widehat\la_\xi$ the $U(1)$-gerbe with connection $(\{\la_{kji}\},\{A_{ji}\},\{B_i+\xi|_{U_i}\})$. Let $E$ be a $\la$-twisted vector bundle and $\Ga=\{\Ga_i\}_{i\in\La}$ be a connection on $E$ associated with $\widehat\la$. We denote the same family $\{\Ga_i\}_{i\in\La}$ on $E$ that is  associated with $\widehat\la_\xi$ as a connection on $E$ by $\Ga_\xi$. We also denote $\xi|_{U_i}$ by $\xi_i$.
\end{nta}

\begin{prp}\label{PRP.change.of.ch.under.change.of.gerbeconnection2forms} Let  $\widehat\la=(\{\la_{kji}\},\{A_{ji}\},\{B_i\})$ be as above, $E$ a $\la$-twisted vector bundle, $\Ga$ a connection on $E$, and $\xi\in\Om^2(X;i\R)$. Then $\ch(\Ga_{-\xi})=\ch(\Ga)\we\exp(\xi).$
\end{prp}
\begin{proof}\beqs
\ch(\Ga_{-\xi})&=\sum_{m=0}^\infty\frac{1}{m!}\ch_{(m)}(\Ga_{-\xi})=\sum_{m=0}^\infty\frac{1}{m!}\tr\left[\left(R_i-B_i\cdot\tbf{1}+\xi_i\cdot\tbf{1}\right)^m\right]\\&=\sum_{m=0}^\infty\frac{1}{m!}\tr\left(\sum_{r=0}^m {{m} \choose {r}} (R_i-B_i\cdot\tbf{1})^{m-r}\xi_i^r\cdot\tbf{1}\right)
\\&=\sum_{m=0}^\infty\frac{1}{m!}\sum_{r=0}^m\left( {{m} \choose {r}}  \tr(R_i-B_i\cdot\tbf{1})^{m-r}\right)\we\xi_i^r\\&=\sum_{m=0}^\infty\frac{1}{m!}\sum_{r=0}^m\left(\frac{m!}{(m-r)!r!} \ch_{(m-r)}(\Ga)\right)\we\xi_i^r=\sum_{m=0}^\infty\sum_{r=0}^m\frac{\ch_{(m-r)}(\Ga)}{(m-r)!}\we\frac{\xi_i^r}{r!}\\
&=\sum_{m=0}^\infty\sum_{r=0}^\infty\frac{\ch_{(m)}(\Ga)}{(m)!}\we\frac{\xi_i^r}{r!}\quad\trm{since $\sum_{m=0}^\infty\sum_{r=0}^m a_{r,m-r}=\sum_{m=0}^\infty\sum_{r=0}^\infty a_{r,m}$}\\&=\ch(\Ga)\we\exp(\xi).
\eeqs
\end{proof}

Now we discuss the Chern-Simons transgression form in the twisted case.

\begin{lem}\label{LEM.path.of.connection} Let  $\widehat\la=(\{\la_{kji}\},\{A_{ji}\},\{B_i\})$ be as above and $\Ga_0$ and $\Ga_1$ be connections on a $\la$-twisted vector bundle $E=(\cU,\{g_{ji}\},\{\la_{kji}\})$ such that both are compatible with $\widehat\la$. Then for each $t\in \R$, $$\Ga_t:=(1-t)\Ga_0+t\Ga_1$$ is a connection on $E$ compatible with $\widehat\la$, i.e., the space of $\widehat\la$-compatible connections on $E$ is an affine space modelled over $\Om^1(X;\trm{End}(E))$.
\end{lem}
\begin{proof}\beqs  {\Ga_t}_i-g_{ji}^{-1}{\Ga_t}_jg_{ji}-g_{ji}^{-1}dg_{ji}&=(1-t){\Ga_0}_i+t{\Ga_1}_i-g_{ji}^{-1}\big((1-t){\Ga_0}_j+t{\Ga_1}_j\big)g_{ji}-g_{ji}^{-1}dg_{ji}\\
&= (1-t){\Ga_0}_i-(1-t)g_{ji}^{-1}{\Ga_0}_jg_{ji}-(1-t)g^{-1}_{ji}dg_{ji}\\&\qquad+t{\Ga_1}_i-tg_{ji}^{-1}{\Ga_1}_jg_{ji}-tg^{-1}_{ji}dg_{ji}\\
&=-(1-t)A_{ji}\cdot\tbf{1}-tA_{ji}\cdot\tbf{1}=-A_{ji}\cdot\tbf{1}.
\eeqs
If $\Ga$ and $\Ga'$ are two different $\widehat\la$-compatible connections on $E$, we have $\Ga_i-\Ga_i'=g_{ji}^{-1}(\Ga_j-\Ga_j')g_{ji}$, so the space of $\widehat\la$-compatible connections on $E$ is an affine space modelled over $\Om^1(X;\trm{End}(E))$. Notice that $\trm{End}(E)$ is an ordinary vector bundle.
\end{proof}

\begin{cor}\label{COR.path.of.bigon} Let $\Ga_0$ and $\Ga_1$ be as in Lemma \ref{LEM.path.of.connection}. Let $\al_t$ and $\ga_t$ with $t\in I$ be two paths of connections each starting at $\Ga_0$ and ending at $\Ga_1$ and both $\al_t$ and $\ga_t$ are compatible with $\widehat\la$ for all $t\in I$. Then there exists a bigon (a polygon with two sides) of connections with edges $\al_t$ and $\ga_t$ such that every point on the bigon is a $\widehat\la$-compatible connection on $E$.
\end{cor}
\begin{proof} By Lemma \ref{LEM.path.of.connection}, for each fixed $t\in I$ and $s\in I$, the connection $(1-s)\al_t+s\ga_t$ is $\widehat\la$-compatible.
\end{proof}

\begin{nta} We shall denote  the projection map $X\times I\ra X$ onto the first factor by $p$.
\end{nta}

\begin{lem} Let  $\widehat\la=(\{\la_{kji}\},\{A_{ji}\},\{B_i\})$ be as above,  $E$ a $\la$-twisted vector bundle $(\cU,\{g_{ji}\},\{\la_{kji}\})$, and $\Ga_t$ a connection on $E$ compatible with $\widehat\la$ for each $t\in I$. Then the family $\{\wtl\Ga_i\}_{i\in \La}$ defined by $\wtl\Ga_i(x,t):=(p^*\Ga_t)(x,t)$ is a connection on  the $(\la\circ p)$-twisted vector bundle $p^*E=(\cU\times I,\{g_{ji}\circ p\},\{\la_{kji}\circ p\})$ compatible with the pull-back $U(1)$-gerbe with connection $p^*\widehat\la=(\{\la_{kji}\circ p\},\{p^*A_{ji}\},\{p^*B_i\}).$
\end{lem}

We refer the reader to Bott and Tu \cite{BT} for an account of the integration along the fiber.

\begin{dfn}\label{DFN.TwistedChernSimons} Let  $\widehat\la=(\{\la_{kji}\},\{A_{ji}\},\{B_i\})$ be as above, $E$ a $\la$-twisted vector bundle over $X$, and $\ga:t\mapsto \Ga_t$ a path of connections on $E$ such that each $\Ga_t$ is compatible with $\widehat\la$. The \textbf{twisted Chern-Simons form} of $\ga$ is the integration along the fiber: $$\trm{cs}(\ga):=\int_I \ch(\wtl\Ga)\in \Om^{\trm{odd}}(X;\C),$$
where $\wtl\Ga$ is the connection on $p^*E$ defined by  $\wtl\Ga(x,t)=(p^*\Ga_t)(x,t)$.
\end{dfn}

The following lemma is certainly well-known, but we did not find a reference.
\begin{lem}\label{LEM.stokes.for.int.along.fiber}
Let $E$ be a smooth fiber bundle over $X$ with fiber $F$ a compact oriented smooth $k$-manifold with corners. Let $\int_F:\Om^\bl(E;\C)\ra \Om^{\bl-k}(X;\C)$ be the integration along the fiber map and $\om\in \Om^n(E;\C)$ for $n\geq k$. Then \beq\label{EQN.Stokes.Int.along.fiber}
d\int_F \om=\int_F d\om+(-1)^{n-k}\int_{\ptl F}\om.\eeq
\end{lem}

\begin{prp}\label{PRP.twisted.Chern-Weil} Let  $\widehat\la=(\{\la_{kji}\},\{A_{ji}\},\{B_i\})$ be as above, $E=(\cU,\{g_{ji}\},\{\la_{kji}\})$ a $\la$-twisted vector bundle of rank $n$, and $\ga:t\mapsto \Ga_t$  a path of connections  on $E$ joining $\Ga_0$ and $\Ga_1$ such that each $\Ga_t$ is compatible with $\widehat\la$. Then $$
\ch(\Ga_0)-\ch(\Ga_1)=(d+H)\trm{cs}(\ga).$$
\end{prp}
\begin{proof} Let $\wtl\Ga$ be a connection on $p^*E$ defined by  $\wtl\Ga(x,t)=(p^*\Ga_t)(x,t)$. By Lemma \ref{LEM.stokes.for.int.along.fiber},
$d\int_I\ch(\wtl\Ga)=\int_I d\ch(\wtl\Ga) -\int_{\ptl I}\ch(\wtl\Ga)=-\int_I \ch(\wtl\Ga)\we p^*H -\int_{\ptl I}\ch(\wtl\Ga)=-H\we\big(\int_I \ch(\wtl\Ga)\big) +\ch(\Ga_0)-\ch(\Ga_1)$. Hence the result.
\end{proof}

\begin{dfn}
The \textbf{twisted total Chern character} of $E$, denoted by $\ch(E)$,  is the twisted cohomology class of $\ch(\Ga)$ for any connection $\Ga$ on $E$.
\end{dfn}

\begin{prp}\label{PRP.evenCh.on.twistedKth} The assignment \beqs \ch:K^0(\cU,\la)&\ra H_{H}^{\trm{even}}(X;\C)\\
[E]-[F] &\mapsto [\ch(\Ga^E)]-[\ch(\Ga^F)],
\eeqs with $(\{A_{ji}\},\{B_i\})$ a  connection  on $\la$ and $\Ga^E$ and $\Ga^F$    connections on $\la$-twisted vector bundles $E$ and $F$, respectively, both compatible with $\widehat\la$, is a well-defined group homomorphism  called the \textbf{twisted Chern character}.
\end{prp}

Before proving Proposition \ref{PRP.evenCh.on.twistedKth}, we recall the following lemma and its generalizations, which are certainly well-known. We include a proof here for sake of completeness. (See also Bunke and Nikolaus \cite{BN}, Section 7).

\begin{lem}\label{LEM.topol.isom.determines.GerbeWithConnectionIso} Suppose $U(1)$-gerbes $\la$ and $\la'$ defined on a good open cover $\cU$ of $X$ are isomorphic: $\la_{kji}'=\la_{kji}+(\de\chi)_{kji}$. Let $(\{A_{ji}\},\{B_i\})$ on $\la$ and $(\{A'_{ji}\},\{B'_i\})$ on $\la'$ be arbitrarily choosen connections. Then there exists  a Deligne $1$-cochain $\widehat\al=(\{\chi_{ji}\},\{\Pi_i\})$ and $\xi\in\Om^2(X;i\R)$ such that $\widehat\la'=\widehat\la_\xi+D\widehat\al$, where $\widehat\la=(\{\la_{kji}\},\{A_{ji}\},\{B_i\})$ and $\widehat\la'=(\{\la_{kji}'\},\{A_{ji}'\},\{B_i'\})$.
\end{lem}

We need the following lemma,   which is well-known, see e.g. Bott and Tu \cite[p.94]{BT}, Proposition 8.5.

\begin{lem}\label{LEM.cech.1.acyclicity} $\check H^p(\cU,\Om^q)=0$, for all $p\geq 1$.
\end{lem}

\begin{proof}[Proof of Lemma \ref{LEM.topol.isom.determines.GerbeWithConnectionIso}] We denote the 3-curvature of $\widehat\la$ and $\widehat\la'$ by $H$ and $H'$, respectively. The curvature $3$-forms of $\widehat\la$ and $\widehat\la'$ are cohomologous, i.e. $H'-H=d\zeta$ for some $\zeta\in\Om^2(X;\sqrt{-1}\R)$. Over each open set $U_i$, we have $d(B_i'-B_i)=d\zeta_{i}$, where $\zeta_{i}:=\zeta|_{U_i}$, and by Poincar\'e's Lemma, there exists a $1$-form $\om_i$ on each $U_i$ such that $B_i'=B_i+\zeta_{i}+d\om_i$. Now by the cocycle condition, $dA_{ji}'=B_j'-B_i'=B_j-B_i+d(\om_j-\om_i)=dA_{ji}+d(\om_j-\om_i)$ and so there exists a $U(1)$-valued function $\mu_{ji}$ over each $U_{ij}$ such that $A_{ji}'=A_{ji}+\om_j-\om_i+d\log\mu_{ji}$. Take the \v{C}ech differential of both sides and get $\de(d\log(\chi\mu^{-1}))_{kji}=0$. By Lemma \ref{LEM.cech.1.acyclicity}, $d\log(\chi_{ji}\mu_{ji}^{-1})=\ga_j-\ga_i$ for some $\ga\in\check{C}^1(\cU,\Om^1)$, hence $d\log(\chi_{ji})=d\log(\mu_{ji})+(\ga_j-\ga_i)$. Notice that the family $\{d\ga_i\}_{i\in\La}$ defines a global $2$-form on $X$. Thus setting $\widehat\al=(\{\chi_{ji}\},\{\om_i-\ga_i\})$ and $\xi|_{U_i}=\zeta_i+d\ga_i$ proves the claim.
\end{proof}

\begin{rmk}\label{RMK.topol.isom.determines.GerbeWithConnectionIso} When the underlying gerbes $\la$ and $\la'$ are identical, a special case of Lemma \ref{LEM.topol.isom.determines.GerbeWithConnectionIso} indicates that, under different choices of connection on $\la$, the corresponding twisted Chern characters are related by $\exp$ of a global $2$-form.
\end{rmk}

\begin{proof}[Proof of Proposition \ref{PRP.evenCh.on.twistedKth}]  \emph{Well-definedness:} Let $\widehat\la$ be fixed. The image of $\ch$ is independent of the choice of connections on  the twisted vector bundles $E$ and $F$ by Proposition \ref{PRP.twisted.Chern-Weil}.

Suppose $\widehat\la$ and $\widehat\la'$ are the same $U(1)$-gerbes $\la$ endowed with different connections and $\widehat\la'=\widehat\la+D\widehat\al$. Then by Proposition \ref{PRP.naturality.of.twist.chern.char.the.same}, the image of $\ch$ is invariant under cohomologous change of $U(1)$-gerbe connection. (Notice that the image of $\ch$ is not invariant under the arbitrary change of $U(1)$-gerbe connection. See Remark \ref{RMK.topol.isom.determines.GerbeWithConnectionIso}.)

Suppose there exists a $\la$-twisted vector bundle $G$ with an isomorphism $\vph:E\dsum G\ra \wbar E\dsum G$. Let $\Ga^G$ be an arbitrary connection on $G$. By Lemma \ref{PRP.ch.is.additive} and Proposition \ref{PRP.naturality.ch.bundleisom}, $$[\ch(\Ga^{E})]+[\ch(\Ga^G)]=[\ch(\Ga^{E}\dsum \Ga^G)]=
[\ch(\vph^*(\Ga^{\wbar E}\dsum \Ga^G))]=[\ch(\Ga^{\wbar E}\dsum \Ga^G)]=[\ch(\Ga^{\wbar E})]+[\ch(\Ga^G)].$$ From this, well-definedness of $\ch$ on  $K^0(\cU,\la)$ follows.

\emph{Group homomorphism:}
This follows from Lemma \ref{PRP.ch.is.additive}:
\beqs \ch([E]-[F]+[\wbar E]-[\wbar F])&=\ch([E\dsum \wbar E]-[F\dsum \wbar F])=
\ch(\Ga^E\dsum \Ga^{\wbar E})-\ch(\Ga^F\dsum \Ga^{\wbar F})\\&=
\ch(\Ga^{E})-\ch(\Ga^{F})+\ch(\Ga^{\wbar E})-\ch(\Ga^{\wbar F})\\&=\ch([E]-[F])+\ch([\wbar E]-[\wbar F]).
\eeqs
\end{proof}

\begin{prp}\label{PRP.cs.for.two.different.paths.is.differ.by.exact} Let  $\widehat\la=(\{\la_{kji}\},\{A_{ji}\},\{B_i\})$ be as above, $E$  a $\la$-twisted vector bundle $(\cU,\{g_{ji}\},\{\la_{kji}\})$, and $\Ga_0$ and $\Ga_1$ two connections on $E$ joined by two different paths of connections $\al_t$ and $\ga_t$ on $E$, such that each of $\al_t$ and $\ga_t$ is compatible with $\widehat\la$ for all $t\in I$. Then $$
\trm{cs}(\ga)-\trm{cs}(\al)\in\trm{Im}(d+H).$$
\end{prp}
\begin{proof} The paths $\al$ and $\ga$ define connections on $p^*E$ over $X\times I$, which we denote by $\wtl\al$ and $\wtl\ga$, respectively. Then there exists a path of connections on $p^*E$ interpolating between $\wtl\al$ and $\wtl\ga$ (by Corollary \ref{COR.path.of.bigon}). Accordingly this path of connection defines a connection $\wtl\be$ on $q^*p^*E$ over $X\times I\times I$, where $q:X\times I\times I \ra X\times I$ is the projection map forgetting the third factor. By applying Lemma \ref{LEM.stokes.for.int.along.fiber} to the twisted Chern character form $\ch(q^*p^*E,\wtl{\wtl{\be}})$, we get
\beqs
d\int_{I\times I} \ch(q^*p^*E,\wtl\be)&=\int_{I\times I} d\ch(q^*p^*E,\wtl\be)+\int_{\ptl(I\times I)}\ch(q^*p^*E,\wtl\be)\\
&=-\Big(\int_{I\times I}\ch(q^*p^*E,\wtl\be)\Big)\we H+\int_{I}\ch(p^*E,{\wtl{\ga}})-\int_{I}\ch(p^*E,{\wtl{\al}})
\eeqs Hence,
$$ \trm{cs}(\ga)-\trm{cs}(\al)=(d+H)\int_{I\times I}  \ch(q^*p^*E,\wtl\be)$$
\end{proof}

\begin{prp}\label{PRP.cs.pullback.by.bundleisom} Let $\vph:E\ra F$ be an isomorphism of $\la$-twisted vector bundles over $X$. Let $\ga:t\mapsto\Ga^t$ be a path of connections on $F$. Then $\cs(\vph^*\ga)=\cs(\ga)$.
\end{prp}

\section{Twisted differential $K$-theory}\label{SEC.4}
This section constitutes the main  part of this paper. We define differential twists and construct a twisted differential $K$-group (Sections \ref{SEC.Differential.twists}, \ref{SEC.Twisted.Differential.K.theory}) using triples consisting of a twisted vector bundle, a connection, and an odd differential form modulo exact forms in a twisted de Rham complex. We verify that our construction is functorial (Section \ref{SEC.functoriality}) and natural with respect to change of twists (Section \ref{SEC.naturality.of.twist}). In sections \ref{SEC.I.and.R.maps} and \ref{SEC.a.map.and.exact.seq}, we define the $I$, $R$,  and $a$ maps and verify the exact sequence involving the $a$ and $I$ maps. Finally, we show commutativity of diagrams and exactness of sequences consisting the hexagon diagram \`a la Simons and Sullivan \cite{SS} (Section \ref{SEC.hexagon.diagram}), and verify that  the maps $I$, $R$, and $a$ are compatible with change of twists (Section \ref{SEC.Compatibility.with.change.of.twist.map}).

\subsection{Differential twists}\label{SEC.Differential.twists}
\begin{dfn} The \textbf{torsion differential $K$-twists} for an open cover $\cU$ of $X$, denoted by $\texttt{Twist}_{\widehat K}^{\trm{tor}}(\cU)$, is the groupoid whose objects are $U(1)$-gerbes with connection  $\widehat\la=(\{\la_{kji}\},\{A_{ji}\},\{B_i\})$ each of which has an underlying $U(1)$-gerbe representing a torsion class in $H^3(X;\Z)$. For any two objects $\widehat\la_1$ and $\widehat\la_2$ in this groupoid, the Hom set is defined by $\textrm{Hom}(\widehat\la_1,\widehat\la_2)
=\{(\widehat\al,\xi)\in \check C^1(\cU;\Om^0)\dsum\check C^0(\cU;\Om^1)\dsum \Om^2(X;i\R)):\widehat\la_2 = \widehat\la_1+D\widehat\al+\xi\}$.
\end{dfn}

\begin{prp}[Existence] Given any manifold $X$ with an open cover $\cU$, the torsion differential twist $\texttt{Twist}_{\widehat K}^{\trm{tor}}(\cU)$ consists of at least one object.
\end{prp}
\begin{proof} The statement amounts to the existence of a connection on a local $U(1)$-bundle gerbe, which follows from the existence of partitions of unity as shown in Murray \cite{Mu}.
\end{proof}

\begin{nta} (1) The \emph{torsion topological $K$-twists} for an open cover $\cU$ of a   manifold $X$ is  the groupoid, denoted by $\texttt{Twist}_{K}^{\trm{tor}}(\cU)$, whose objects are $U(1)$-gerbes $\la=\{\la_{kji}\}$ each representing a torsion class in $H^3(X;\Z)$. A morphism from $\la_1$ to $\la_2$ is a \v{C}ech $1$-cochain $\al=(\chi_{ji})\in \check C^1(\cU,U(1))$ such that $\la_1=\la_2+\de\al$.

(2) Define the groupoid $\Om^3_{\trm{cl}}(X;i\R)$ of $i\R$-valued closed differential $3$-forms on $X$ as follows. Objects are $i\R$-valued closed differential $3$-forms on $X$. A morphism from $\om$ to $\om'$ is a differential $2$-form $\al$ on $X$ modulo exact forms satisfying $\om=\om'+d\al$.
\end{nta}

\begin{dfn}The \textbf{forgetful} and \textbf{curvature} functors are given by the assignments \beqs
\cF:\texttt{Twist}_{\widehat K}^{\trm{tor}}(\cU)&\ra  \texttt{Twist}_{K}^{\trm{tor}}(\cU)&\trm{Curv}:\texttt{Twist}_{\widehat K}^{\trm{tor}}(\cU)&\ra \Om_{\trm{cl}}^3(X;i\R)\\
\widehat\la=(\{\la_{kji}\},\{A_{ji}\},\{B_i\}) &\mapsto \{\la_{kji}\}&\widehat\la=(\{\la_{kji}\},\{A_{ji}\},\{B_i\}) &\mapsto \trm{Curv}(\widehat\la)=H,\\
(\widehat\al,\xi) &\mapsto \al &(\widehat\al,\xi) &\mapsto d\xi \\
\eeqs where $\widehat\al=(\{\chi_{ji}\},\{\Pi_i\}), \al=(\{\chi_{ji}\})$, and $H|_{U_i}=dB_i$ for all $i\in\La$.
\end{dfn}

\begin{rmk}
Let $f:(Y,\cV) \ra (X,\cU)$ be a   map with $\cV=f^{-1}\cU$. The following diagrams commute:
\[\xymatrix{
\texttt{Twist}_{\widehat K}^{\trm{tor}}(\cV)\ar[r]^\cF & \texttt{Twist}_{K}^{\trm{tor}}(\cV) & & \texttt{Twist}_{\widehat K}^{\trm{tor}}(\cV)\ar[r]^{\trm{Curv}} & \Om_{\trm{cl}}^3(Y;i\R)\\
\texttt{Twist}_{\widehat K}^{\trm{tor}}(\cU)\ar[u]^{f^*} \ar[r]^{\cF}  & \texttt{Twist}_{K}^{\trm{tor}}(\cU) \ar[u]^{f^*} & & \texttt{Twist}_{\widehat K}^{\trm{tor}}(\cU)\ar[u]^{f^*} \ar[r]^{\trm{Curv}}  & \Om_{\trm{cl}}^3(X;i\R) \ar[u]^{f^*}
}\] Here $f^*$ on torsion differential twists takes each torsion $U(1)$-gerbe with connection to its pullback $U(1)$-gerbe with pullback connection, and $f^*$ on topological twists does the same on torsion $U(1)$-gerbes.
\end{rmk}

\begin{nta}
Throughout this section, we shall use the notation $\widehat\la$ to denote a differential twist $(\{\la_{kji}\},\{A_{ji}\},\{B_i\})\in \texttt{Twist}_{\widehat K}^{\trm{tor}}(\cU)$, $H$ for $\trm{Curv}(\widehat\la)$, and $\la$ for $\cF(\widehat\la)$.
\end{nta}

\subsection{Twisted differential $K$-group}\label{SEC.Twisted.Differential.K.theory}
\begin{dfn} A \textbf{$\widehat K^0(\cU;\widehat\la)$-generator} is a triple $(E,\Ga,\om)$ consisting of a $\la$-twisted vector bundle $E$ defined on  the open cover $\cU=\{U_i\}_{i\in\La}$ on $X$, a connection $\Ga$ on $E$ compatible with $\widehat\la$, and $\om\in \Om^{\trm{odd}}(X;\C)/\trm{Im}(d+H)$.
\end{dfn}

\begin{dfn}\label{DFN.capital.CS} Let $E$ be any $\la$-twisted vector bundle with a path of connections $\ga$ joining $\Ga_0$ and $\Ga_1$ and each connection on the path  being compatible with $\widehat\la$. Define
$\CS(\Ga_0,\Ga_1):=\cs(\ga) \mod\trm{Im}(d+H)$.
\end{dfn}

\begin{rmk}\label{RMK.CS.insertion.of.connection} By Proposition \ref{PRP.cs.for.two.different.paths.is.differ.by.exact}, Definition \ref{DFN.capital.CS} is independent of the choice of path of connections. Furthermore, we have $\CS(\Ga_0,\Ga_1)+\CS(\Ga_1,\Ga_2)=\CS(\Ga_0,\Ga_2)$.
\end{rmk}

\begin{dfn}\label{DEF.triplerel} Two $\widehat K^0(\cU;\widehat\la)$-generators $(E,\Ga,\om)$ and $(E',\Ga',\om')$ are \textbf{equivalent} if there exists a $\la$-twisted vector bundle with connection $(F,\Ga^F)$ and a $\la$-twisted vector bundle isomorphism $\vph=\{\vph_i\}_{i\in \La}:E\dsum F\ra E'\dsum F$ such that $\CS(\Ga\dsum\Ga^F,\vph^*(\Ga'\dsum\Ga^F))=\om-\om'$.
\end{dfn}

\begin{lem} The relation between triples in Definition \ref{DEF.triplerel} is an equivalence relation.
\end{lem}
\begin{proof} The relation is reflexive since $\cs$ of a loop is $(d+H)$-exact. For symmetry, suppose $(E,\Ga,\om)$ and $(E',\Ga',\om')$ are equivalent, i.e., there exists a $\la$-twisted vector bundle with connection $(F,\Ga^F)$ whose connection is compatible with $\widehat\la$ such that there is an isomorphism of $\la$-twisted vector bundles $\vph:E\dsum F \srl{\isom}\ra E'\dsum F$, and $\CS(\Ga\dsum \Ga^F,\vph^*(\Ga'\dsum \Ga^F))=\om-\om'$. By Proposition \ref{PRP.cs.pullback.by.bundleisom},
$$\CS(\Ga\dsum \Ga^F,\vph^*(\Ga'\dsum \Ga^F))=\CS((\vph^{-1})^*(\Ga\dsum \Ga^F),\Ga'\dsum \Ga^F).$$ This proves the symmetry. For transitivity, suppose $(E,\Ga,\om)$ is equivalent to $(E',\Ga',\om')$ and $(E',\Ga',\om')$ is equivalent to $(E'',\Ga'',\om'')$, i.e., there exists a $\la$-twisted vector bundle with connection $(F,\Ga^F)$ whose connection is compatible with $\widehat\la$ such that there is an isomorphism of $\la$-twisted vector bundles $\vph:E\dsum F \srl{\isom}\ra E'\dsum F$, and $\CS(\Ga\dsum \Ga^F,\vph^*(\Ga'\dsum \Ga^F))=\om-\om'$, and there exists a $\la$-twisted vector bundle with connection $(F',\Ga^{F'})$ whose connection is compatible with $\widehat\la$ such that there is an isomorphism of $\la$-twisted vector bundles $\vph':E'\dsum F' \srl{\isom}\ra E''\dsum F'$, and $\CS(\Ga'\dsum \Ga^{F'},{\vph'}^*(\Ga''\dsum \Ga^{F'}))=\om'-\om''$. Then by taking the $\la$-twisted vector bundle with connection $(F\dsum F',\Ga^{F}\dsum \Ga^{F'})$, the isomorphism of $\la$-twisted vector bundles $\psi:E\dsum F\dsum F'\ra E''\dsum F\dsum F'$ is defined by the composition$$
E\dsum F\dsum F'\srl{\vph\dsum\tbf{1}}\lra E'\dsum F\dsum F' \srl{\tbf{1}\dsum\si}\lra E'\dsum F'\dsum F \srl{\vph'\dsum\tbf{1}}\lra E''\dsum F'\dsum F \srl{\tbf{1}\dsum\si^{-1}}\lra E''\dsum F\dsum F',$$ each of which is an isomorphism, and $\si$ is the canonical $\la$-twisted vector bundle isomorphism $F\dsum F'\ra F'\dsum F$. Furthermore, \beqs &\CS(\Ga\dsum \Ga^F\dsum \Ga^{F'},\psi^*(\Ga''\dsum \Ga^F\dsum \Ga^{F'}))\\&=\CS(\Ga\dsum \Ga^F\dsum \Ga^{F'},(\vph\dsum\tbf{1})^*(\Ga'\dsum \Ga^F\dsum \Ga^{F'}))\\&\qquad\qquad+\CS((\vph\dsum\tbf{1})^*(\Ga'\dsum \Ga^F\dsum \Ga^{F'}),\psi^*(\Ga''\dsum \Ga^F\dsum \Ga^{F'})) \quad\trm{by Remark \ref{RMK.CS.insertion.of.connection}}\\
&\srl{*}=\om-\om'+\CS(\Ga'\dsum \Ga^F\dsum \Ga^{F'},\big((\tbf{1}\dsum\si^{-1})\circ (\vph'\dsum \tbf{1}) \circ(\tbf{1}\dsum\si)\big)^*(\Ga''\dsum \Ga^{F}\dsum \Ga^{F'}))\\
&\srl{**}=\om-\om'+\CS(\Ga'\dsum \Ga^{F'}\dsum \Ga^{F},\big((\tbf{1}\dsum\si^{-1})\circ (\vph'\dsum \tbf{1}) \big)^*(\Ga''\dsum \Ga^{F}\dsum \Ga^{F'}))\\
&=\om-\om'+\CS(\Ga'\dsum \Ga^{F'}\dsum \Ga^{F}, (\vph'\dsum \tbf{1})^*(\Ga''\dsum \Ga^{F'}\dsum \Ga^{F}))\\&=\om-\om'+\om'-\om''=\om-\om'',
\eeqs
At $*$ and $**$, we have used Proposition \ref{PRP.cs.pullback.by.bundleisom} for the twisted bundle isomorphism $(\vph\dsum\tbf{1})^{-1}$ and $(\tbf{1}\dsum\si^{-1})^{-1}$, respectively. Hence $(E,\Ga,\om)$ is equivalent to $(E'',\Ga'',\om'')$.
\end{proof}

\begin{lem} Let $[(E,\Ga^E,\om)]$ and $[(F,\Ga^F,\eta)]$ be equivalence classes of $K^0(\cU;\widehat\la)$-generators. The equivalence class of the $K^0(\cU;\widehat\la)$-generator $(E\dsum F,\Ga^E\dsum\Ga^F,\om+\eta)$ is independent of the choice of representatives of $[(E,\Ga^E,\om)]$ and $[(F,\Ga^F,\eta)]$.
\end{lem}


\begin{dfn} The \textbf{addition} $+$ between two equivalence classes of $\widehat K^0(\cU;\widehat\la)$-generators is defined by  $[(E,\Ga^E,\om)]+[(F,\Ga^F,\eta)]:=[(E\dsum F,\Ga^E\dsum\Ga^F,\om+\eta)]$.
\end{dfn}

Hence the set of all equivalence classes of $\widehat K^0(\cU;\widehat\la)$-generators forms a commutative monoid $(\f{G},+)$.

\begin{dfn} \label{DEF.BSKtype.twisted.diff.even.K} Let  $\widehat\la\in\texttt{Twist}_{\widehat K}^{\trm{tor}}(\cU)$. The \textbf{twisted differential $K$-group} is $$\widehat K^0(\cU,\widehat\la):=K(\f{G}),$$ where $K$ denotes the group completion functor from commutative monoid to abelian groups.
\end{dfn}

\subsection{Functoriality}\label{SEC.functoriality}
\begin{lem}\label{LEM.functoriality} Let $\widehat\la$ be a differential twist, $E=(\cU,\{g_{ji}\},\{\la_{kji}\})$ a $\la$-twisted vector bundle of rank $n$, and $\Ga=\{\Ga_i\}_{i\in\La}$ connection on $E$ compatible with $\widehat\la$. Also let $f:(Y,\cV) \ra (X,\cU)$ be a   map with $\cV=f^{-1}\cU$. If two triples $(E,\Ga,\om)$ and $(E',\Ga',\om')$ are equivalent, then $(f^*E,f^*\Ga,f^*\om)$ and $(f^*E',f^*\Ga',f^*\om')$ are equivalent.
\end{lem}

\begin{prp}\label{LEM.functoriality.of.tw.diff.K}
Given a   map $f:(Y,\cV) \ra (X,\cU)$ with $\cV=f^{-1}\cU$, the assignment \beqs f^*:\widehat K^0(\cU,\widehat\la)&\ra \widehat K^0(\cV,f^*\widehat\la)\\
[(E,\Ga^E,\om)]-[(F,\Ga^F,\eta)] &\mapsto [(f^*E,f^*\Ga^E,f^*\om)]-[(f^*F,f^*\Ga^F,f^*\eta)]
\eeqs is a well-defined group homomorphism.
\end{prp}

Let \textbf{Man} be the category whose objects are connected compact smooth manifolds equipped with an open cover. A morphism from $(Y,\cV)$ to $(X,\cU)$ is a smooth map satisfying $\cV=f^{-1}(\cU)$. Also let \textbf{Ab} be the category of abelian groups.
\begin{cor} $\widehat K^0(-,\widehat\la):\tbf{Man}^{\trm{op}}\ra \tbf{Ab}$ is a functor.
\end{cor}

\subsection{Naturality of twists}\label{SEC.naturality.of.twist}
\begin{prp}\label{PRP.naturality.of.twists} Let $\widehat\la=(\{\la_{kji}\},\{A_{ji}\},\{B_i\})$ and $\widehat\la'=(\{\la'_{kji}\},\{A'_{ji}\},\{B'_i\})$ be any two objects of $\texttt{Twist}_{\widehat K}^{\trm{tor}}(\cU)$ satisfying that $\widehat\la'=\widehat\la+D\widehat\al$ for some $\widehat\al=(\{\chi_{ji}\},\{\Pi_i\})$. Let $E=(\cU,\{g_{ji}\},\{\la_{kji}\})$ be a $\la$-twisted vector bundle of rank $n$ and $\Ga=\{\Ga_i\}_{i\in\La}$ a connection on $E$ compatible with $\widehat\la$. Define:\beq \label{EQN.naturality.of.twists}
E' &:= (\cU,\chi_{ji}g_{ji},\la'_{kji})\\
\Ga' &:= \{\Ga'_i\}_{i\in\La}\trm{ where }\Ga'_i:=\Ga_i+\Pi_i\cdot\textbf{1}\\
\om' &:=\om
\eeq (1) The assignment
\beqs \phi_{\widehat\al}:\widehat K^0(\cU;\widehat\la)&\srl{\isom}\ra \widehat K^0(\cU;\widehat\la')\\ [(E,\Ga,\om)]-[(F,\na,\eta)]&\mapsto [(E',\Ga',\om')]-[(F',\na',\eta')]
\eeqs is an induced group isomorphism that is natural in $\cU$.

(2) Let $\xi\in\Om^2(X;i\R)$. The assignment
\beqs \Xi:\widehat K^0(\cU;\widehat\la)&\srl{\isom}\ra \widehat K^0(\cU;\widehat\la_\xi)\\ [(E,\Ga,\om)]-[(F,\na,\eta)]&\mapsto [(E,\Ga_\xi,\om\we\exp(-\xi))]-[(F,\na_\xi,\eta\we\exp(-\xi))]
\eeqs is a group isomorphism that is natural in $\cU$.
\end{prp}
\begin{rmk} The family $\Ga'$ above is a connection on  the $\la'$-twisted vector bundle $E'$ compatible with $\widehat{\la}'$: \beqs
& g_{ji}^{-1}\chi_{ji}^{-1}\Ga_j'\chi_{ji}g_{ji}+g_{ji}^{-1}\chi_{ji}^{-1}d(\chi_{ji}g_{ji})-A_{ji}'\cdot\tbf{1}\\&\qquad=g_{ji}^{-1}\Ga_jg_{ji}+\Pi_j\cdot\tbf{1}+\chi_{ji}^{-1}d\chi_{ji}\cdot\tbf{1}+g_{ji}^{-1}dg_{ji}-(A_{ji}+\Pi_j-\Pi_i+d\log\chi_{ji})\cdot\tbf{1}\\&\qquad=\Ga_i+\Pi_i\cdot\tbf{1}=\Ga_i'.
\eeqs
\end{rmk}
\begin{proof}[Proof of Proposition \ref{PRP.naturality.of.twists}] (1) Suppose $(E,\Ga,\om)\sim(\wbar E,\wbar \Ga,\wbar \om)$, i.e., there exists a twisted vector bundle $F$ and a connection $\Ga^F$ compatible with $\widehat\la$ and a $\la$-twisted vector bundle isomorphism $\vph=\{\vph_i\}:E\dsum F\ra \wbar E\dsum F$ such that $$\CS(\Ga\dsum\Ga^F,\vph^*(\wbar\Ga\dsum \Ga^F))=\om-\wbar\om.$$ We verify that $(E',\Ga',\om')$ and $(\wbar E',\wbar \Ga',\wbar \om')$ are equivalent so that well-definedness of the map follows. We take a $\la'$-twisted vector bundle $F'$ and a connection $\Ga'^F$ compatible with $\widehat\la'$ by applying the same rule in \eqref{EQN.naturality.of.twists} to $(F,\Ga^F)$. There exists a $\la'$-twisted vector bundle isomorphism $\vph=\{\vph_i\}:E'\dsum F'\ra \wbar E'\dsum F'$ defined exactly the same as the above $\vph$.\footnote{Let $\wbar g_{ji}$, $h_{ji}$ be transition maps of $\wbar E$ and $F$, respectively. Since $\vph$ is an isomorphism, we have $$\vph_j(x)\big(\wbar g_{ji}(x)\dsum h_{ji}(x)\big)=\big(\wbar g_{ji}(x)\dsum h_{ji}(x)\big)\vph_i(x)$$ for all $x\in U_{ij}$. From this we have $$\vph_j(x)\big(\wbar g_{ji}(x)\chi_{ji}(x)\dsum h_{ji}(x)\chi_{ji}(x)\big)=\big(\wbar g_{ji}(x)\chi_{ji}(x)\dsum h_{ji}(x)\chi_{ji}(x)\big)\vph_i(x)$$ for all $x\in U_{ij}$.} We have to show that  $$\CS(\Ga'\dsum\Ga'^F,\vph^*(\wbar\Ga'\dsum \Ga'^F))=\om'-\wbar\om'.$$

Suppose $\wtl\Ga$ is a connection on $p^*(E\dsum F)$ over $X\times I$ defined by a path of connections joining $\Ga\dsum\Ga^F$ and $\vph^*(\wbar\Ga\dsum \Ga^F)$ on $E\dsum F$ over $X$. By definition, $$\cs(\Ga^t):=\int_I \ch(\wtl \Ga)=\int_I \trm{rank}(E\dsum F)+\sum_{m=1}^\infty\frac{1}{m!}\tr(\wtl R_i-p^*B_i\cdot\tbf{1})^m.$$

We see that:\beq\label{EQN.naturality.of.twist.welldef} \CS(\Ga'\dsum\Ga'^F,\vph^*(\wbar\Ga'\dsum \Ga'^F))&=\CS(\Ga\dsum\Ga^F+\Pi\cdot\tbf{1},\vph^*(\wbar\Ga\dsum \Ga^F+\Pi\cdot\tbf{1}))\\
&=\CS(\Ga\dsum\Ga^F+\Pi\cdot\tbf{1},\vph^*(\wbar\Ga\dsum \Ga^F)+\Pi\cdot\tbf{1})\\
&=\cs(\Ga^t+\Pi\cdot\tbf{1})\mod\trm{Im}(d+H)\\
&=\int_I \ch(\wtl \Ga+p^*\Pi\cdot\tbf{1})\mod\trm{Im}(d+H).
\eeq Since \beqs
d\big(\wtl \Ga_i+p^*\Pi_i\cdot\tbf{1}\big)+\big(\wtl \Ga_i+p^*\Pi_i\cdot\tbf{1}\big)\we \big(\wtl \Ga_i+p^*\Pi_i\cdot\tbf{1}\big)= \wtl R_i + p^*d\Pi_i\cdot\tbf{1},
\eeqs we have
\beqs
\int_I \ch(\wtl \Ga+p^*\Pi\cdot\tbf{1}) &= \int_I \trm{rank}(E'\dsum F')+\sum_{m=1}^\infty\frac{1}{m!}\tr(\wtl R_i+p^*d\Pi_i\cdot\tbf{1}-p^*B'_i\cdot\tbf{1})^m\\
&= \int_I \trm{rank}(E\dsum F)+\sum_{m=1}^\infty\frac{1}{m!}\tr(\wtl R_i-p^*B_i\cdot\tbf{1})^m\\
&=\int_I \ch(\wtl \Ga).
\eeqs

Hence the far RHS of \eqref{EQN.naturality.of.twist.welldef} is \beqs
\int_I \ch(\wtl \Ga)\Big/\trm{Im}(d+H)&=\CS(\Ga\dsum\Ga^F,\vph^*(\wbar\Ga\dsum \Ga^F))=\om-\wbar \om.
\eeqs

The map $\phi_{\widehat\al}$ being bijective, a group homomorphism and natural in $\cU$ is straightforward.

(2) We first show that, if $(E,\Ga,\om)\sim (\wbar E,\wbar \Ga,\wbar \om)$, then  $(E,\Ga_\xi,\om\we\exp(-\xi))\sim (\wbar E,\wbar \Ga_\xi,\wbar \om\we\exp(-\xi))$. By the premise, there exists a $\la$-twisted vector bundle $G$ and a connection $\Ga^G$ on $G$ compatible with $\widehat\la$ and an isomorphism $\vph:E\dsum G\ra \wbar E \dsum G$, such that $\om-\wbar\om=\CS(\Ga\dsum\Ga^G,\vph^*(\wbar\Ga\dsum\Ga^G))=\int_I \ch(\wtl\Ga) \mod\trm{Im}(d+H)$, where $\wtl\Ga$ is a connection on $p^*(E\dsum G)$ defined by pullback of connections on a straight line path joining $\Ga\dsum\Ga^G$ and $\vph^*(\wbar\Ga\dsum\Ga^G)$. Accordingly, \beqs
\CS(\Ga_\xi\dsum\Ga_\xi^G,\vph^*(\wbar \Ga_\xi\dsum\Ga_\xi^G))&=\int_I\big(\ch(\wtl\Ga)\we\exp(-p^*\xi)\big)\mod\trm{Im}(d+H+d\xi)\\
&=\Big(\int_I\ch(\wtl\Ga)\Big)\we\exp(-\xi)\mod\trm{Im}(d+H+d\xi)\\&= (\om-\wbar\om)\we\exp(-\xi) \mod\trm{Im}(d+H+d\xi).
\eeqs From this, well-definedness of the map $\Xi$ follows.
The map $\Xi$ being one-to one, onto, group homomorphism, and  natural in $\cU$ are all obvious.
\end{proof}

\subsection{The $I$ and $R$ map}\label{SEC.I.and.R.maps}
\begin{prp} Let $\widehat\la\in\texttt{Twist}_{\widehat K}^{\trm{tor}}(\cU)$. The assignment
\beqs
I:\widehat K^0(\cU,\widehat\la)&\ra K^0(\cU,\la)\\
[(E,\Ga^E,\om)]-[(F,\Ga^F,\eta)] &\mapsto [E]-[F]
\eeqs is a group homomorphism which is natural in $\cU$.
\end{prp}

\begin{prp} Let $\widehat\la\in\texttt{Twist}_{\widehat K}^{\trm{tor}}(\cU)$. The assignment  \beqs
R:\widehat K^0(\cU,\widehat\la)&\ra \Om^{\trm{even}}(X;\C)\\
[(E,\Ga^E,\om)]-[(F,\Ga^F,\eta)] &\mapsto \ch(\Ga^E)+(d+H)\om-\ch(\Ga^F)-(d+H)\eta
\eeqs is a group homomorphism which is natural in $\cU$.
\end{prp}

\subsection{Odd twisted Chern character forms}\label{APP.OddTwistedChernChar}
In this subsection, we define odd twisted Chern character forms which will be used in Sections \ref{SEC.a.map.and.exact.seq}, \ref{SEC.hexagon.diagram}, and \ref{SEC.Compatibility.with.change.of.twist.map}.

\begin{dfn}\label{DEF.oddtwistedchernchar} Let $X$ be a  manifold, $\cU$  a good open cover of $X$, $\widehat\la=(\{\la_{kji}\},\{A_{ji}\},\{B_i\})$ a $U(1)$-gerbe with connection on $\cU$ whose Dixmier-Douady class is torsion. Also let $E$ be a $\la$-twisted vector bundle,  $\phi$ an automorphism on $E$, and $\Ga$ a connection compatible with $\widehat\la$. The \textbf{total twisted odd Chern character form} of the triple $(E,\phi,\Ga)$ is $\cs\left(t\mapsto(1-t)\Ga^E+t\phi^*\Ga^E\right)$.
\end{dfn} By Propositons \ref{PRP.twisted.Chern-Weil} and \ref{PRP.naturality.ch.bundleisom}, $\cs\left(t\mapsto(1-t)\Ga^E+t\phi^*\Ga^E\right)$ represents an odd twisted cohomology class.

The odd twisted Chern character form is functorial.
\begin{prp} Given a map $f:(Y,\cV)\ra (X,\cU)$ with $\cV=f^{-1}(\cU)$, the following holds: $$\Ch(f^*E,(\phi\circ f), f^*\Ga)=f^*\Ch(E,\phi,\Ga)$$
\end{prp}
\begin{proof} Note that $f^*\phi^*\Ga^E=(\phi\circ f)^{-1}\circ f^*\Ga^E\circ (\phi\circ f)+(\phi\circ f)^{-1}d(\phi\circ f)=(\phi\circ f)^*\Ga^E$. The proof of this statement is similar to the proof of Lemma \ref{LEM.functoriality}.
\end{proof}

The total odd twisted Chern character form respects change of differential twist in a manner that is similar to the even case. (Compare Propositions \ref{PRP.naturality.of.twist.chern.char.the.same} and \ref{PRP.change.of.ch.under.change.of.gerbeconnection2forms}.)

\begin{prp}\label{PRP.oddChernChar.naturality.of.twist} (1) Let  $\widehat\la$, $\widehat\la'$, $E$, $E'$, $\Ga$, and $\Ga'$ be as in Proposition \ref{PRP.naturality.of.twist.chern.char.the.same}. The following holds: $$\Ch(E',\phi,\Ga')=\Ch(E,\phi,\Ga)$$

(2) Let $\xi\in\Om^2(X;i\R)$, and $\widehat\la$ and $\widehat\la_\xi$ be as in Notation \ref{NTA.change.of.curving.by.a.global.2.form}. The following holds: $$\Ch(E,\phi,\Ga_\xi)=\Ch(E,\phi,\Ga)\we\exp(-\xi)$$
\end{prp}
\begin{proof} (1) $\Ch(E,\phi,\Ga')=\cs(t\mapsto (1-t)\Ga'+t\phi^*\Ga')=\cs(t\mapsto (1-t)\Ga+t\phi^*\Ga+\Pi\cdot\tbf{1})=\cs(t\mapsto (1-t)\Ga+t\phi^*\Ga)=\Ch(E,\phi,\Ga)$, where the third equality follows from a similar calculation appearing in the proof of Proposition \ref{PRP.naturality.of.twists}.

(2) Let $\wtl\Ga$ be a connection on $p^*E$ defined by pullback of connections on the path $(1-t)\Ga^E+t\phi^*\Ga^E$. We have $\Ch(E,\phi,\Ga^E_\xi)=\cs(t\mapsto (1-t)\Ga_\xi^E+t\phi^*\Ga_\xi^E)=\int_I\ch(\wtl\Ga)\we\exp(-p^*\xi)=(\int_I\ch(\wtl\Ga))\we\exp(-\xi)=\Ch(E,\phi,\Ga^E)\we\exp(-\xi)$.
\end{proof}

\subsection{The $a$ map and the exact sequence involving the $a$ and $I$ maps}\label{SEC.a.map.and.exact.seq}

\begin{nta}\label{NTA.AllChForms.and.AllExactForms}  We denote by $\Om_{H,\Ch}$ the abelian group $\trm{Im}(\Ch)+\trm{Im}(d+H)$, where $\trm{Im}(\Ch)$ is the abelian group generated by all odd twisted Chern character forms.
\end{nta}


\begin{dfn}\label{DEF.the.a.circle.map} Let $\widehat\la\in \texttt{Twist}_{\widehat K}^{\trm{tor}}(\cU)$. Define:\beqs
a:\Om^{\trm{odd}}(X;\C)/\Om_{H,\Ch} &\ra \widehat K^0(\cU,\widehat\la)\\
\te &\mapsto [(\mathscr{O},0,\te)].
\eeqs
\end{dfn}

\begin{lem}\label{LEM.a.circ.well.def} The map $a$ is well-defined, a group homomorphism and is natural in $\cU$.
\end{lem}
\begin{proof}
We prove well-definedness. \beqs a(\te+\Ch(E,\phi,\Ga^E))&=[(\mathscr{O},0,\te+\Ch(E,\phi,\Ga^E))]\\
=&[(\mathscr{O}\dsum E,0\dsum \Ga^E,\te+\cs((1-t)\Ga^E+t\phi^*\Ga^E))]-[(E,\Ga^E,0)]\\
=&[(\mathscr{O},0,\te)]+[(E,\Ga^E,\cs((1-t)\Ga^E+t\phi^*\Ga^E))]-[(E,\Ga^E,0)]\\
=&[(\mathscr{O},0,\te)]=a(\te).\eeqs
\end{proof}

\begin{prp}The following sequence is exact:
$$0\ra \Om^{\trm{odd}}(X;\C)/\Om_{H,\Ch} \srl{a}\ra \widehat K^0(\cU,\widehat\la)\srl{I}\ra K^0(\cU,\la)\ra 0.$$
\end{prp}
\begin{proof} It is obvious that $I$ is surjective, and $\trm{Im }a\subseteq\ker I$. We show the other inclusion.

Let $[(E,\Ga^E,\om)]-[(F,\Ga^F,\eta)]\in \widehat K^0(\cU,\widehat\la)$ and suppose $I\big([(E,\Ga^E,\om)]-[(F,\Ga^F,\eta)]\big)=0$. Then there exists a $\la$-twisted vector bundle $G$ and an isomorphism of $\la$-twisted vector bundles $\vph:E\dsum G\ra F\dsum G$. Choose any connection $\Ga^G$ on $G$ that is compatible with $\widehat\la$. Then \beqs\ptm
[(E,\Ga^E,\om)]-[(F,\Ga^F,\eta)]&=[(E\dsum G,\Ga^E\dsum \Ga^G,\om)]-[(F\dsum G,\Ga^F\dsum\Ga^G,\eta)]\\
&\srl{*}=[(E\dsum G,\Ga^E\dsum \Ga^G,\om)]-[(E\dsum G,\Ga^E\dsum\Ga^G,\mu)],
\eeqs where $\mu:=\eta+\CS(\Ga^E\dsum\Ga^G,\vph^*(\Ga^F\dsum\Ga^G))$, and $*$ follows from the fact that $(E\dsum G,\Ga^E\dsum\Ga^G,\mu)$ is equivalent to $(F\dsum G,\Ga^F\dsum\Ga^G,\eta)$. We now add and subtract $(\mathscr{O},0,0)$, and get:
\beqs\ptm \qquad&[(E\dsum G,\Ga^E\dsum \Ga^G,\om)]-[(E\dsum G,\Ga^E\dsum\Ga^G,\mu)]\\&=[(E\dsum G\dsum \mathscr{O},\Ga^E\dsum \Ga^G\dsum 0,\om)]-[(E\dsum G\dsum \mathscr{O},\Ga^E\dsum\Ga^G\dsum 0,\mu)]\\
&=[(E\dsum G,\Ga^E\dsum \Ga^G,0)]+[(\mathscr{O},0,\om)]-[(E\dsum G,\Ga^E\dsum\Ga^G,0)]-[(\mathscr{O},0,\mu)]\\&=[(\mathscr{O},0,\om-\mu)]=a(\om-\mu).
\eeqs
We now show injectivity of the $a$ map. Suppose $a(\te)=0$. Equivalently, there exists a $\la$-twisted vector bundle with connection $(E,\Ga^E)$ whose connection is compatible with $\widehat\la$ and an automorphism $\vph\in \textrm{Aut}(E)$ satisfying that $\te=\CS(\Ga^E,\vph^*\Ga^E)=\Ch(E,\vph,\Ga^E)\mod \trm{Im}(d+H)$  Hence the result.
\end{proof}

\begin{nta}
Let $\cU:=\{U_i\}_{i\in \La}$ be an open cover of a space $X$. Recall that a \emph{refinement} of $\cU$ is a pair $(\cV,\tau)$ consisting of an open cover $\cV:=\{V_r\}_{r\in \cL}$ of $X$ and a map $\tau:\cL\ra \La$ such that $V_r\subset U_{\tau(r)}$ for all $r\in \cL$. It induces the following restriction maps on the totality of $U(1)$-gerbes with connection, $\la$-twisted vector bundles, and the space of connections on a $\la$-twisted vector bundle $E$, respectively.
\beqs
\widehat\la=(\la_{kji},A_{ji},B_i) &\mapsto \tau^*\widehat\la=(\tau^*\la_{tsr},\tau^*A_{sr},\tau^*B_r):=(\la_{\tau(t)\tau(s)\tau(r)},A_{\tau(s)\tau(r)},B_{\tau(r)})\\
E=(\cU,\{g_{ji}\},\{\la_{kji}\}) &\mapsto \tau^*E:=(\cV,\{g_{\tau(s)\tau(r)}\},\{\tau^*\la_{tsr}\})\\
\Ga=\{\Ga_i\}_{i\in\La}&\mapsto\tau^*\Ga=\{(\tau^*\Ga)_r:=\Ga_{\tau(r)}\}_{r\in\cL}
\eeqs
\end{nta}

\begin{cor} Let $\cU:=\{U_i\}_{i\in \La}$ be a good cover of $X$ and $(\cV,\tau)$ consisting of an open cover $\cV:=\{V_r\}_{r\in \cL}$ of $X$ and a map $\tau:\cL\ra \La$ be a choice of refinement of $\cU$. The restriction map induced by $\tau$ \beqs
\widehat R_\tau:\widehat K^0(\cU;\widehat\la)&\ra\widehat K^0(\cV;\tau^*\widehat\la)\\
[(E,\Ga^E,\om)]-[(F,\Ga^F,\eta)]&\mapsto [(\tau^*F,\tau^*\Ga^E,\om)]-[(\tau^*F,\tau^*\Ga^F,\eta)]
\eeqs is an isomorphism of abelian groups.
\end{cor}
\begin{proof} Consider the following diagram.
\[\xymatrix{
0 \ar[r]\ar@{=}[d]& \Om^{\text{odd}}(X;\C)/\Om_{H,\Ch} \ar@{=}[d]\ar[r]^{\qquad a}& \widehat{K}^0(\cU,\widehat{\la}) \ar[d]^{\widehat R_\tau}\ar[r]^{I}&   K^0(\cU;\la) \ar[d]_{\isom}^{R_\tau}\ar[r]& 0\ar@{=}[d]\\
0 \ar[r]& \Om^{\text{odd}}(X;\C)/\Om_{H,\Ch} \ar[r]^{\qquad a}& \widehat{K}^0(\cV,\tau^*\widehat{\la}) \ar[r]^{I}  & K^0(\cV;\tau^*\la) \ar[r]& 0
}
\] The map $R_\tau$ is an isomorphism (See Karoubi \cite[p.233]{Ka2} Theorem 3.6), and all square diagrams commute. Hence $\widehat R_\tau$ is an isomorphism by the five-lemma.
\end{proof}

\begin{dfn} Let $\underline{\widehat\la}$ be the element in the colimit of Deligne $2$-cocycles over $X$ defined on a good open cover $\cU$ along a choice of refinements of $\cU$ represented by $\widehat\la\in\texttt{Twist}_{\widehat K}^{\trm{tor}}(\cU)$.  The \textbf{twisted differential $K$-group} of $X$, denoted by $\widehat K^0(X,\underline{\widehat\la})$, is defined by the colimit of $\widehat K^0(\cU,\widehat\la)$ over all refinements of $\cU$. The \textbf{twisted $K$-group} of $X$, denoted by $K^0(X,\underline{\la})$, is defined by the colimit of $K^0(\cU,\la)$ over all refinements of $\cU$.
\end{dfn}

\subsection{The hexagon diagram}\label{SEC.hexagon.diagram}
\begin{nta}\label{NTA.hexagon.subsection} We denote by $\trm{Pr}:\Om^{\trm{even}}(X;\C)_{\trm{closed}}\ra H_H^{\trm{even}}(X;\C)$ the map taking twisted de Rham cohomology class, and $r:H_H^{\trm{odd}}(X;\C)\ra\Om^{\trm{odd}}(X;\C)/\Om_{H,\Ch}$ the map that sends an odd twisted de Rham cohomology class $[\om]$ to $\om+\Om_{H,\Ch}$. The map $r$ is well-defined by definition of $\om+\Om_{H,\Ch}$ (see Notation \ref{NTA.AllChForms.and.AllExactForms}).
\end{nta}

\begin{prp}\label{PRP.hexagon.1} For the maps $I$, $R$, and $a$ from or into $\widehat K^0(\cU;\widehat\la)$, the following holds:\begin{itemize}
    \item[(1)] $\ch\circ I=\trm{Pr}\circ R$.
    \item[(2)] $R\circ a=d+H$.
\end{itemize}
\end{prp}

\begin{dfn} We define maps $\al$ and $\be$ as follows:\beqs &\al: H_H^{\odd}(X;\C)\ra \ker R&&\be:\ker R\ra K^{0}(\cU,\la)\\
&[\om]\mapsto (\mathscr{O},0,\om)\qquad&& (E,\Ga^E,\om)-(F,\Ga^F,\eta)\mapsto [E]-[F].
\eeqs
\end{dfn}
\begin{rmk} The maps $\al$ and $\be$ are well-defined group homomorphisms.
\end{rmk}

\begin{prp}
\begin{itemize}
    \item[(1)] $a \circ r=\trm{incl}\circ \al$.
    \item[(2)] $\be =I\circ \trm{incl}$.
    \item[(3)] The following sequences are exact:
\beqs
H_H^{\trm{odd}}(X;\C) \srl{\al}\ra &\ker R \srl{\be}\ra K^{0}(\cU,\la) \srl{\ch}\ra H_H^{\trm{even}}(X;\C)\\
H_H^{\trm{odd}}(X;\C) \srl{r}\ra&\Om^{\trm{odd}}(X;\C)/\trm{Im}(d+H) \srl{d+H}\ra \Om^{\trm{even}}(X;\C)_{\trm{closed}} \srl{\trm{Pr}}\ra H_H^{\trm{even}}(X;\C)
\eeqs
\end{itemize}
\end{prp}
\begin{proof} All claims are obvious except that $\ker(\be)\subseteq\trm{Im}(\al)$, which we prove presently. Take an arbitrary element $[(E,\Ga^E,\om)]-[(F,\Ga^F,\eta)]\in \ker R$ whose image under $\be$ is zero, i.e., there exists a $\la$-twisted vector bundle $G$ defined on $\cU$ and an isomorphism $\vph:E\dsum G\srl{\isom}{\ra}F\dsum G$. Choose any connection $\Ga^G$ on $G$ that is compatible with $\widehat\la$. At this point we introduce the following notation: $$\zeta:= \CS\left(\Ga^E\dsum \Ga^G,\vph^*(\Ga^F\dsum\Ga^G)\right).$$

We see that
\beq\label{EQN.UpperLES}\ptm
[(E,\Ga^E,\om)]-[(F,\Ga^F,\eta)]&=[(E,\Ga^E,\om)]-[(F,\Ga^F,\om-\zeta)]+[(\mathscr{O},0,\om-\zeta-\eta)]\\&=a([\om-\zeta-\eta]),\eeq where in the first equality we add and subtract $[(\mathscr{O},0,\om-\zeta-\eta)]$ and in the second use Definition \ref{DEF.triplerel}.

We have to verify that the differential form $\om-\eta-\zeta$ represents an odd degree twisted cohomology class. Since $[(E,\Ga^E,\om)]-[(F,\Ga^F,\eta)]\in \ker R$, we have $\ch(\Ga^E)-\ch(\Ga^F)+(d+H)(\om-\eta)=0$. Now
$(d+H)(\om-\eta-\zeta)=\ch(\Ga^E)-\ch(\Ga^F)-(\ch(\Ga^E\dsum \Ga^G)-\ch(\Ga^F\dsum\Ga^G))=0$.
\end{proof}

\begin{cor}\label{COR.hexagon.diagram} In the following diagram for $\widehat K^0(X;\underline{\widehat\la})$, all square and triangles are commutative and all sequences are exact.

\[
\xy0;/r.25pc/:
(-30,30)*+{0}="0TL";
(-30,0)*+{H^{\mathrm{odd}}_{H}(X;\mathbb{C})}="L";
(-15,15)*+{\mathrm{ker}(R)}="TL";
(-30,-30)*+{0}="0BL";
(-15,-15)*+{\Omega^{\mathrm{odd}}(X)/\Omega_{H,\mathrm{Ch}}}="BL";
(0,0)*+{\widehat K^{0}(X;\underline{\widehat\la})}="M";
(15,15)*+{K^{0}(X,\underline\lambda)}="TR";
(30,30)*+{0}="0TR";
(15,-15)*+{\mathrm{Im}(R)}="BR";
(30,0)*+{H^{\mathrm{even}}_{H}(X;\mathbb{C})}="R";
(30,-30)*+{0}="0BR";
(-15,0)*{\scriptstyle\circlearrowleft};
(0,9)*{\scriptstyle\circlearrowleft};
(0,-8)*{\scriptstyle\circlearrowright};
(15,0)*{\scriptstyle\circlearrowleft};
{\ar"0TL";"TL"};
{\ar"L";"TL"^{\alpha}};
{\ar"TL";"TR"^{\beta}};
{\ar@{}"TL";"M"^(0.2){}="a"^(0.8){}="b"};
{\ar@{^{(}->}"a";"b"};
{\ar"M";"TR"^{I}};
{\ar"TR";"0TR"};
{\ar"TR";"R"^{\mathrm{ch}}};
{\ar"L";"BL"^{r}};
{\ar"0BL";"BL"};
{\ar"BL";"M"^{a}};
{\ar"BL";"BR"^(0.6){d+H}};
{\ar@{->>}"BR";"R"};
{\ar"BR";"0BR"};
{\ar"M";"BR"^{R}};
\endxy
\]
\end{cor}

\begin{rmk} When the differential twist is $\widehat\la=(\{1\},\{0\},\{0\})$, the diagram reduces to the differential $K$-theory hexagon diagram of Simons and Sullivan (see \cite[p. 596]{SS}).
\end{rmk}

\subsection{Compatibility with change of twist map}\label{SEC.Compatibility.with.change.of.twist.map}

\begin{prp} Let $\widehat\al'=(\chi'_{ji},\Pi'_i)$ be an isomorphism $\widehat\la\ra\widehat\la'$ such that $\widehat\la'=\widehat\la+D\al'$. Then the diagram in Corollary \ref{COR.hexagon.diagram} is natural under change of twist by $\widehat\al'$:  \begin{itemize}
\item[(1)] $I_{\widehat\la'}\circ \phi_{\widehat\al'}=\phi_{\al'} \circ I_{\widehat\la}$.
\item[(2)] $\phi_{\widehat\al'}\circ a_{\widehat\la} =a_{\widehat\la'}$.
\item[(3)] $R_{\widehat\la'}\circ \phi_{\widehat\al'}=R_{\widehat\la}$.
\end{itemize}
\end{prp}
\begin{proof} (1) \beqs I_{\widehat\la'}\circ \phi_{\widehat\al'}([(E,\Ga^E,\om)]-[(F,\Ga^F,\eta)])&=I_{\widehat\la'}([(\phi_{\al'}E,\Ga^E+\Pi'\cdot\tbf{1},\om)]-[(\phi_{\al'}F,\Ga^F+\Pi'\cdot\tbf{1},\eta)])\\&=[\phi_{\al'}E]-[\phi_{\al'}F]\\&=\phi_{\al'} \circ I_{\widehat\la}([(E,\Ga^E,\om)]-[(F,\Ga^F,\eta)]).\eeqs

(2) Obvious.

(3) \beqs
&R_{\widehat\la'}\circ \phi_{\al'}\big([(E,\Ga^E,\om)]-[(F,\Ga^F,\eta)]\big)=R_{\widehat\la}\big([(\phi_{\al'}E,\{\Ga^E_i+\Pi_i'\},\om)]-[(\phi_{\al'}F,\{\Ga^F_i+\Pi_i'\},\eta)]\big)\\
&\qquad=\big(\ch(\Ga^E+\Pi'\cdot\tbf{1})+(d+H)\om-\ch(\Ga^F+\Pi'\cdot\tbf{1})-(d+H)\eta\big)\\
&\qquad=\big(\ch(\Ga^E)+(d+H)\om-\ch(\Ga^F)-(d+H)\eta\big) \qquad\trm{by Lemma \ref{PRP.naturality.of.twist.chern.char.the.same}.}\\
&\qquad=R_{\widehat\la}\big([(E,\Ga^E,\om)]-[(F,\Ga^F,\eta)]\big).\eeqs
\end{proof}

\begin{prp} Let $\xi\in\Om^2(X;i\R)$ that induces an isomorphism $\widehat\la\ra \widehat\la_\xi$. Then the diagram in Corollary \ref{COR.hexagon.diagram} is natural under change of twist by $\xi$:
\begin{itemize}
    \item[(1)] $\Xi \circ a=a \circ \exp(-\xi)$.
    \item[(2)] $I\circ \Xi=I$.
    \item[(3)] $R \circ \Xi=\exp(-\xi) \circ R$.
\end{itemize}
\end{prp}

\end{document}